\documentclass[11pt, reqno]{amsart}

\usepackage{amssymb,latexsym}

\usepackage{enumerate}

\newtheorem{thm}{ \bf Theorem}[section]
\newtheorem{cor}[thm]{ \bf Corollary}
\newtheorem{lem}[thm]{ \bf Lemma}
\newtheorem{prop}[thm]{ \bf Proposition}

\newtheorem{rem}[thm]{ \bf Remark}

\numberwithin{equation}{section}

\begin{document}

\baselineskip=17pt

\title[ Risk-sensitive control ]
{ Risk-sensitive control of continuous time \\Markov chains}
\thanks{This work is supported in part by SPM fellowship of
CSIR and in part by UGC Centre for Advanced Study.}

\author[Ghosh]{ Mrinal K. Ghosh}
\address{Department of Mathematics\\
Indian Institute of Science\\
Bangalore 560 012, India.}
\email{mkg@math.iisc.ernet.in}

\author[Saha]{ Subhamay Saha}
\address{Department of Mathematics\\
Indian Institute of Science\\
Bangalore 560 012, India.}
\email{subhamay@math.iisc.ernet.in}


\date{}

\begin{abstract}
We study risk-sensitive control of continuous time Markov chains taking values in discrete state space. We study both finite and infinite horizon problems. In the finite horizon problem we characterise the value function via HJB equation and obtain an optimal Markov control. We do the same for infinite horizon discounted cost case. In the infinite horizon average cost case we establish the existence of an optimal stationary control under certain Lyapunov condition. We also develop a policy iteration algorithm for finding an optimal control.
\end{abstract}





\maketitle

\section{\textbf{Introduction and Preliminaries}}
In the last two decades considerable attention has been given to the investigation of risk
sensitive problems in the literature of stochastic dynamic optimization. An important
reason for the popularity of this kind of problems is its connections with H$_{\infty}$ or
robust control problems and stochastic dynamic games. A justification for the term risk-sensitive
control comes from utility theory in economics. Generally in stochastic dynamic optimization,
the decision maker (controller) seeks to minimise a cost functional which is a random quantity,
say, $X$, which depends on the time horizon and the control adopted by the controller.
Since $X$ is random the controller tries to minimise the expected value of $X$.
This is the risk neutral case. But this approach has some limitations namely if
the variance is large then there can be issues with the optimal control.
Generally variance is a measure of risk in economics literature. So ideally one would
like to minimise both mean and variance simultaneously, but this may not be feasible.
Therefore a convex combination of mean and variance is optimised or the mean is optimised
for a given variance. This approach of mean- variance optimization was taken by Markowitz
in his work on portfolio selection \cite{Markowitz}. This was later extended by Sharpe in his
capital asset pricing model \cite{Sharpe}. But if the random variable is not normally distributed,
then its distribution is not completely determined by the first two moments. Thus it is reasonable
to consider a cost criterion which deals with higher moments as well. A powerful approach in this
direction is the risk-sensitive control wherein the controller seeks to minimise an exponential criterion.
Roughly speaking the cost functional of interest is of the form $\mathbb{E}\exp(\theta X)$ where $X$ is
the random variable which denotes the cost payable by the controller and $\theta > 0$ is a parameter
chosen by the controller and whose interpretation is given below. Let $w$ be the amount the controller is
willing to pay instead of the random quantity $X$. Thus $w$ satisfies $$\exp(\theta w)=\mathbb{E}\exp(\theta X)\,.$$
The deterministic quantity $w$ is referred to as the certainty equivalent of $X$.
The risk premium $\pi$ is defined by the equation $$w = \mathbb{E}X + \pi\,.$$
Now by Jensen's inequality $$\exp(\theta \mathbb{E}X) \leq \mathbb{E}\exp(\theta X) = \exp(\theta w).$$
Thus by the monotonicity property of the exponential function, $w \geq \mathbb{E}X$, which implies $\pi \geq 0$.
Thus in this case the controller is risk averse. Now to measure the degree of risk aversion,
let $x = \mathbb{E}X$. Formally by Taylor's expansion $$\exp(\theta w)= \exp(\theta x) +
\pi \theta\exp(\theta x) + o(\pi)\,.$$ Again $$\mathbb{E}\exp(\theta X) = \exp(\theta x) + \frac{1}{2}\mbox{var}(X)
\theta^2 \exp(\theta x)+ \mathbb{E}(o(X-x)^2)\,.$$ Thus we have $\pi=\frac{1}{2}\mbox{var}(X) \theta$ plus smaller order terms.
Hence the risk premium is proportional to $\theta$ up to to first order.
That is why $\theta$ is referred to as the absolute risk aversion parameter.
Similar arguments can also be made for $\theta < 0$ case. In that case the controller is risk seeking.
The limiting case of $\theta = 0$ is the risk neutral case.

There is a vast literature on the risk neutral case, for example
see \cite{Ari} for controlled diffusions, \cite{hgbook} for
continuous time MDP, \cite{Lerma} for discrete time MDP and the
references therein. See also \cite{PR} for variance minimization
and overtaking optimality of continuous-time MDP. For earlier
works on risk-sensitive control we refer to \cite{Howard} and
\cite{Jacobson}. Since then there has been a lot of research on
risk senstive control of discrete time Markov chains
\cite{Borkar}, \cite{Chung}, \cite{Hernandez}, \cite{JasKiewicz}
\cite{dimasi} and also there has been a lot of work on risk
sensitive control of diffusions \cite{Biswas}, \cite{Fleming},
\cite{Menaldi}, \cite{Nagai}, \cite{Whittle}. As is evident from
the discussion above, risk sensitive control has wide applications
in economics and in particular in finance \cite{Goswami},
\cite{Pliska1}, \cite{Pliska2}, \cite{Sobel}.

Although risk sensitive control of continuous time diffusions and discrete time
Markov chains has been studied, the problem for continuous time MDP does not seem to have been studied in literature. In this paper we study risk-sensitive control of continuous time Markov chains. We take the state space $S$ to be countable. For notational simplicity we take $S=\{0,1,2,\dots\}$.
Let $U_i, i=0,1,\cdots$ be compact metric spaces; $U_i$ is the control set when the state is $i$.
We denote the state process by $\{X_t\}$ and the control process by $\{U_t\}$.
Formally the dynamics of the process is as follows:
\begin{align}\begin{cases}\label{dynamics}&\mathbb{P}(X_{t+h}=j~|~X_t=i,U_t=u)= \lambda_{ij}(u)h + o(h)\\&
\mathbb{P}(X_{t+h}=i~|~X_t=i,U_t=u)= 1- \bigl(\sum_{j \neq i}\lambda_{ij}(u)\bigr)h + o(h)\,,
\end{cases}
\end{align}where $\lambda_{ij}: U_i \rightarrow \mathbb{R}_+$ are given functions.
That is, if the process is at $i$ at time $t$ and if the action chosen at that moment is $u$,
then after a little while $h$ the process will be at state $j$ with probability $\lambda_{ij}(u)h$
plus some error term and the process will remain at $i$ with probability
$1- \bigl(\sum_{j \neq i}\lambda_{ij}(u)\bigr)h$ plus some error term.
Thus $\lambda_{ij}$s are the instantaneous transition rates. Set
\begin{eqnarray}\label{conv}\lambda_{ii}(u)= - \sum_{j \neq i}\lambda_{ij}(u)\,.\end{eqnarray}The following assumptions will be in force throughout the paper:

\noindent \textbf{(A1)} The function $\lambda_{ij}$s are continuous.

\noindent \textbf{(A2)} $\displaystyle\sup_{i}\sup_{u\in U_i}\{-\lambda_{ii}(u)\} \leq M < \infty$.

\noindent \textbf{(A3)} The sum in \eqref{conv} converges uniformly. Thus $\lambda_{ii}$ is continuous for each $i$.

We now describe a rigorous construction of the process $\{X_t\}$ via the martingale problem.
Let $\mathcal{D}=\mathcal{D}([0,\infty), S)$ be the space of $S-$valued right-continuous functions
with left limits endowed with the Skorokhod topology. Let $\mathcal{S}$ be the Borel $\sigma-$algebra on $\mathcal{D}$.
Define $U = \displaystyle\cup_{i} U_i$. Let $\textbf{u}: [0,\infty)\times S \rightarrow U$ be
such that $\textbf{u}(.,i)\in U_i$ and is measurable for each $i$. Let $B(S)$ denote the space of
bounded real valued functions on S. For $f \in B(S)$, $||f||$ denotes the supremum norm.
For each $t \in [0,\infty)$ define the operator $\Lambda^{\textbf{u}}_t:B(S)\rightarrow B(S)$ by
\begin{eqnarray}\label{generator} \Lambda^{\textbf{u}}_tf(i)= \sum_j \lambda_{ij}(\textbf{u}(t,i))f(j)\,.\end{eqnarray}
On the measurable space $(\mathcal{D},\mathcal{S})$, let $\{X_t,t\geq 0\}$ denote the canonical process,
i.e., for $\omega \in \mathcal{D}$, $X_t(\omega)=\omega(t)$. Let $\mu$ be any probability measure on $S$.
The martingale problem corresponding to $(\Lambda^{\textbf{u}},\mu)$ is the following:
A measure $\mathbb{P}^{\textbf{u}}_{s,\mu}$ on $(\mathcal{D},\mathcal{S})$ is said to be a
solution for the martingale problem corresponding to $(\Lambda^{\textbf{u}},\mu)$ if \\i) $\mathbb{P}^{\textbf{u}}_{s,\mu}(X_s\in A)=\mu(A)$
for any Borel subset $A$ of $S$;
\\ ii) $f(X_t)- \int_0^t \Lambda^{\textbf{u}}_sf(X_s)ds$ is a $\mathbb{P}^{\textbf{u}}_{s,\mu}$
martingale with respect to the filtration $\mathcal{F}_t=\sigma(X_r;r\leq t)$ for each $f$ in $B(S)$.
\\Under \textbf{(A2)} it can be shown following the arguments in Chapter $6$ of \cite{Ethier}
that the above martingale problem has a unique solution and $\{X_t\}$ is a Markov process with the
generator given by \eqref{generator}. In fact we can relax the boundedness condition in \textbf{(A2)}.
If $\lambda_{ii}$s satisfy appropriate growth condition then also the martingale problem is well posed;
see Chapter $6$ of \cite{Ethier}. Also see \cite{Prieto} and the references therein for related works.
From now on we will work in the canonical space $(\mathcal{D},\mathcal{S})$. If $s=0$ and $\mu =\delta_i$ for some $i \in S$ then we will write $\mathbb{P}^{\textbf{u}}_{s,\mu}$ as $\mathbb{P}^{\textbf{u}}_i$. The corresponding expectation operator is denoted by $\mathbb{E}_i^{\textbf{u}}$.
In our paper the set of admissible controls is the set of Markov controls, i.e., controls of the form $U_t = \textbf{u}(t, X_{t-})$, for some $\textbf{u}: [0,\infty)\times S \rightarrow U$, such that $\textbf{u}(.,i)\in U_i$ and is measurable for each $i$. With an abuse of terminology the map \textbf{u} itself is referred to as a Markov control. Let $\mathcal{U}$ denote the set of all Markov controls. A Markov control is said to be stationary if the function \textbf{u} has no explicit dependence on $t$, i.e., $\textbf{u}: S \rightarrow U$, such that $\textbf{u}(i)\in U_i$ for each $i$. The set of stationary Markov controls is denoted by $\mathcal{U}_s$.

Now we briefly describe the problems we consider in this paper.
In stochastic dynamic optimization based on the time horizon there can be two kinds of problems
namely finite horizon and infinite horizon problems. In this paper we address both infinite and finite horizon problems. \\
\textbf{Finite Horizon Problem:} Define $K=\{(i,u): i\in S, u \in U_i\}$. Let $c:[0,\infty)\times K \rightarrow [0,\infty)$ be a bounded function, such that $c(.,i,.)$ is continuous for each $i$ and $g:S \rightarrow [0,\infty)$ a bounded function. Let $0<T<\infty$ be the length of the time horizon. Then for any Markov control $U$ consider the cost functional
\begin{eqnarray}J^{\textbf{u}}_T(i)=\frac{1}{\theta}\log\mathbb{E}_{i}^{\textbf{u}}\biggl[\exp\bigl(\theta\bigl[\int_0^Tc(s,X_s,U_s)ds + g(X_T)\bigl]\bigl)\biggr]
\end{eqnarray} for some $\theta \in (0,1)$ and where $U_t=\textbf{u}(t,X_{t-})$. In literature $c$ is referred to as the the running cost function and $g$ as the terminal cost function. The aim of the controller is to minimise $J^{\textbf{u}}_T$ over all Markov controls $\textbf{u}$. A control $\hat{\textbf{u}}$ is said to be optimal if $$ J^{\hat{\textbf{u}}}_T(i)= \displaystyle \inf_{\mathcal{U}} J^{{\textbf{u}}}_T(i)\,.$$ \\
\textbf{Infinite Horizon Discounted Cost Problem:} For the infinite horizon problems the running cost function has no explicit time dependence. For each Markov control $\textbf{u}$ define
\begin{eqnarray}\label{discounted}I_{\alpha}(\theta , i, \textbf{u}) = \frac{1}{\theta}\log \biggl(E_i^{\textbf{u}}\biggl(\exp \biggl[\theta \int_0^{\infty}e^{-\alpha t}c(X_t,U_t)dt\biggr]\biggr)\biggr)\end{eqnarray} where $\theta$ is as in the finite horizon problem and $\alpha>0$ is the discount factor. Here the controller wants to minimise $I_{\alpha}(\theta , i, \textbf{u})$ over all Markov controls \textbf{u}. A control $\hat{\textbf{u}}$ is said to be optimal if it satisfies
$$I_{\alpha}(\theta , i, \hat{\textbf{u}})=\displaystyle \inf_{\mathcal{U}} I_{\alpha}(\theta , i, \textbf{u})\,.$$
\textbf{Infinite Horizon Average Cost Problem:} For the average cost problem the set of
admissible controls is the set of stationary Markov controls. For a stationary control \textbf{u}, define
\begin{eqnarray}\label{objective}J^{\textbf{u}}(i) = \limsup_{T \rightarrow \infty}\frac{1}{T}
\log \mathbb{E}_i^{\textbf{u}}\biggl[ \exp \biggl ( \theta \int_0^T c(X_t,\textbf{u}(X_t))dt\biggr)\biggr]\,.\end{eqnarray}
The controller wants to minimise $J^{\textbf{u}}(i)$ over all stationary controls \textbf{u}. Optimal control is defined analogously.\\
The rest of the paper is organised as follows. In Section 2 we
study the finite horizon problem. This analysis of this problem is
fairly straightforward. Using the dynamic programming heuristics
we derive the Hamilton Jacobi Bellman (HJB) equation for this
criteria. Then using a fixed point theorem and some standard
arguments involving Dynkin's formula we show that the value
function is the unique solution of the HJB equation in an
appropriate class of functions. This in turn yields the existence
of an optimal Markov control. Section 3 deals with infinite
horizon discounted cost case. The analysis of this problem is
surprisingly far more involved from a technical view point. As
usual by using the dynamic programming heuristics we derive the
HJB equation and establish the corresponding verification theorem.
However, to establish the existence of a smooth solution of the
HJB equation for this criteria turns out to be quite tricky. We
work around this problem by an appropriate limiting procedure
which establishes a solution of the HJB equation in the sense of
distributions. Then under certain assumptions we establish the
desired regularity of the solution. In Section 4 we investigate
the average cost problem. Again this problem turns out to be
technically involved. The traditional vanishing discount approach
does not seem to work. Instead we use the multiplicative Poisson
equation to get at the desired result. Using a limiting argument
involving the multiplicative Poisson equation we establish the
existence of an optimal stationary control. In Section 5 we give a
policy improvement algorithm for the average cost case. Finally in
Section 6 we conclude our paper with some concluding remarks.

\section{\textbf{Finite Horizon Case}}
In this section we study the finite horizon case. For this we first study the
exponential cost criterion. For $t \in [0,T]$, $\textbf{u}\in \mathcal{U}$, define
\begin{eqnarray}\label{finitehorizon}\hat{J}^{\textbf{u}}_T(t,i)=\mathbb{E}_{t,i}
^{\textbf{u}}\biggl[\exp\bigl(\theta\bigl[\int_t^Tc(s,X_s,U_s)ds + g(X_T)\bigl]\bigl)\biggr]\,.
\end{eqnarray}Define the value function $V_T$ by
$$V_{T}(t,i)=\displaystyle \inf_{\mathcal{U}}\hat{J}^{\textbf{u}}_T(t,i)\,$$ where the
infimum is over all Markov controls. Our aim is to characterise the value function and to
obtain an optimal control. To this end we first describe a heuristic derivation of the Hamilton Jacobi Bellman (HJB) equation.
Formally
{\small\begin{align*}V_T(t, i)&= \inf_{\mathcal{U}}\mathbb{E}_{t,i}^{\textbf{u}}
\biggl\{\exp\biggl[ \theta \int_t^{t+h} c(s,X_s,U_s)ds+\theta \int_{t+h}^T c(s,X_s,U_s)ds+\theta g(X_T) \biggr]\biggr\}
\\&= \inf_{\mathcal{U}}\mathbb{E}_{t,i}^{\textbf{u}}\biggl\{\exp\biggl[ \theta \int_t^{t+h}c(s,X_s,U_s)ds\biggr]
\mathbb{E}_{t+h, X_{t+h}}^{\textbf{u}}\biggl(\exp\biggl[\theta \int_{t+h}^T c(s,X_s,U_s)ds + \theta g(X_T)\biggr]\biggl)\biggr\}\\
&=\inf_{\mathcal{U}}\mathbb{E}_{t,i}^{\textbf{u}}\biggl\{\exp\biggl[ \theta \int_t^{t+h}c(s,X_s,U_s)ds\biggr]V_{T}(t+h ,X_{t+h})\biggr\}\,.
\end{align*}}
If the function $V_T(.,i)$ is continuously differentiable then standard dynamic programming arguments involving Dynkin's
formula leads to the following HJB equation for the finite horizon problem:
\begin{align}\label{fhdpe}\begin{cases}\frac{d \varphi}{dt}+\displaystyle \inf_{u\in U_i}\bigl[\theta c(t,i,u)\varphi(t,i)+\sum_{j}\lambda_{ij}(u)\varphi(t,j)\bigr]=0\,\,\mbox{on}\,\,[0,T) \times S\\
\varphi(T,i)=e^{\theta g(i)}\,.
\end{cases}
\end{align}The importance of this equation is highlighted by the following verification theorem:
\begin{thm}Assume \textbf{(A1)-(A3)}. Suppose there exists a smooth (continuously differentiable with respect to the first variable), bounded solution $\Psi$ to $\eqref{fhdpe}$, then $$\Psi(t,i)=V_T(t,i) \,\,\mbox{for ~all}\,\,(t,i)\,\,\in [0,T]\times S\,.$$ Furthermore an optimal Markov control for the cost criterion \eqref{finitehorizon} exists and is given by $U^*_t= \textbf{u}^*(t,X_{t-})$
where $\textbf{u}^*$ satisfies
\begin{eqnarray}\label{policy}\nonumber \inf_{u\in U_i}\bigl[\theta c(t,i,u)\Psi(t,i)+\sum_{j}\lambda_{ij}(u)\Psi(t,j)\bigr]
\\=\bigl[\theta c(t,i,\textbf{u}^*(t,i))\Psi(t,i)+\sum_{j}\lambda_{ij}(\textbf{u}^*(t,i))\Psi(t,j)\bigr]\,.
\end{eqnarray}
\end{thm}
\begin{proof}Let \textbf{u} be any arbitrary Markov control. By Feynman - Kac formula
{\small\begin{align*}&\mathbb{E}_{t,i}^{\textbf{u}}\biggl[\exp\bigl(\theta\bigl[\int_t^Tc(s,X_s,U_s)ds + g(X_T)\bigl]\bigl)\biggr]\\&= \Psi(t,i) + \mathbb{E}_{t,i}^{\textbf{u}}\int_t^T\exp\biggl(\theta\bigl[\int_t^rc(s,X_s,U_s)ds\bigr]\biggl)\biggl[\frac{d \Psi}{dr}(r,X_r)+ \theta c(r,X_t,U_t)\Psi(r,X_r)+\sum_{j}\lambda_{X_rj}(U_r)\Psi(t,j)\biggr]dr\,.
\end{align*}}Since $\Psi$ satisfies \eqref{fhdpe} we have
$$\Psi(t,i)\leq \mathbb{E}_{t,i}^{\textbf{u}}\biggl[\exp\bigl(\theta\bigl[\int_t^Tc(s,X_s,U_s)ds + g(X_T)\bigl]\bigl)\biggr]\,.$$
Now if we use the control $\textbf{u}^*$ as in \eqref{policy} then we get an equality in the above in place of inequality. The existence of an $\textbf{u}^*$ satisfying \eqref{policy} is ensured by a standard measurable selection theorem \cite{benes}. Hence the theorem follows.
\end{proof}
Next we prove that there exists a smooth, bounded solution to $\eqref{fhdpe}$.
\begin{thm}Assume \textbf{(A1)-(A3)}. Then there exists a unique solution to $\eqref{fhdpe}$ in $C_b([0,T]\times S)\bigcap C^1([0,T)\times S)$.
\end{thm}
\begin{proof}Let $\varphi(t,i)=e^{-\gamma_0 t}\psi(t,i)$. Substituting in $\eqref{fhdpe}$ we get
\begin{align}\label{fhdpe1}\begin{cases}\frac{d \psi}{dt}-\gamma_0 \psi+\displaystyle \inf_{u\in U_i}\bigl[\theta c(t,i,u)\psi(t,i)+\sum_{j}\lambda_{ij}(u)\psi(t,j)\bigr]=0\,\,\mbox{on}\,\,[0,T) \times S\\
\psi(T,i)=e^{\gamma_0 T}e^{\theta g(i)}\,.
\end{cases}
\end{align}Consider the following integral equation:
\begin{eqnarray}\label{fixedpoint}\psi(t,i)= e^{\gamma_0 t}e^{\theta g(i)}+e^{\gamma_0 t}\int_t^Te^{-\gamma_0 s}\displaystyle \inf_{u\in U_i}\bigl[\theta c(s,i,u)\psi(s,i)+\sum_{j}\lambda_{ij}(u)\psi(s,j)\bigr]ds \,.
\end{eqnarray}Define $\mathcal{T}:C_b([0,T]\times S)\rightarrow C_b([0,T]\times S)$ by
\begin{eqnarray*}\mathcal{T}\psi(t,i)=e^{\gamma_0 t}e^{\theta g(i)}+e^{\gamma_0 t}\int_t^Te^{-\gamma_0 s}\displaystyle \inf_{u\in U_i}\bigl[\theta c(s,i,u)\psi(s,i)+\sum_{j}\lambda_{ij}(u)\psi(s,j)\bigr]ds\,.
\end{eqnarray*}
Then
\begin{align*}|\mathcal{T}\psi_1(t,i)-\mathcal{T}\psi_2(t,i)|&\leq e^{\gamma_0 t}\int_{t}^T e^{-\gamma_0 s}\bigl\{\theta ||c||||\psi_1-\psi_2||+2M||\psi_1-\psi_2||\bigr\}ds\\&=(2M+ \theta ||c||)||\psi_1-\psi_2||e^{\gamma_0 t}\int_t^T e^{-\gamma_0 s}ds\\
&=\frac{2M+ \theta ||c||}{\gamma_0}||\psi_1-\psi_2||e^{\gamma_0 t}(e^{-\gamma_0 t}-e^{-\gamma_0 T})\\
&\leq \frac{2M+ \theta ||c||}{\gamma_0}||\psi_1-\psi_2||\,,
\end{align*}where $M$ is as in \textbf{(A2)}. Hence for $\gamma_0=2M + \theta ||c||+1$, $\mathcal{T}$ is a contraction and thus by Banach's fixed point theorem there exists a unique solution to $\eqref{fixedpoint}$ in $C_b([0,T]\times S)$. Using \textbf{(A1)-(A3)}, the boundedness and continuity of the cost function $c$, it follows that $\psi$ is in $C_b([0,T]\times S)\bigcap C^1[0,T)\times S$. Then it follows that $\varphi(t,i)=e^{-\gamma_0 t}\psi(t,i)$ is a solution of \eqref{fhdpe}. The uniqueness follows from the previous theorem.
\end{proof}Thus combining the above two theorems we get the following theorem:
\begin{thm}Under \textbf{(A1)-(A3)}, the value function $V_T$ is the unique solution to $\eqref{fhdpe}$ in $C_b([0,T]\times S)\bigcap C^1([0,T)\times S$). An optimal control is given by the Markov control $U^*_t=\textbf{u}^*(t,X_{t-})$ where $\textbf{u}^*$ satisfies
\begin{eqnarray}\label{control}\nonumber\inf_u\bigl[\theta c(t,i,u)V_T(t,i)+\sum_{j}\lambda_{ij}(u)V_T(t,j)\bigr]
\\=\bigl[\theta c(t,i,\textbf{u}^*(t,i))V_T(t,i)+\sum_{j}\lambda_{ij}(\textbf{u}^*(t,i))V_T(t,j)\bigr]\,.
\end{eqnarray}\end{thm}
Now since logarithm is an increasing function the following theorem is now evident.
\begin{thm} Let $\varphi$ be the unique solution of \eqref{fhdpe} in $C_b([0,T]\times S)\bigcap C^1([0,T)\times S$). Define $\psi=\theta^{-1}\log\varphi$. Then
$$\psi(t,i)=\inf_{\mathcal{U}}\frac{1}{\theta}\log\mathbb{E}_{i}^{\textbf{u}}\biggl[\exp\bigl(\theta\bigl[\int_0^Tc(s,X_s,U_s)ds + g(X_T)\bigl]\bigl)\biggr]\,.$$ Moreover the Markov control given by \eqref{control} is again an optimal control in this case.
\end{thm}

\section{\textbf{Discounted Cost Case}} In this section we turn our attention towards infinite horizon discounted cost problem.
Define
\begin{eqnarray}\label{infinitehorizon}V_{\alpha}(\theta,i)=
\displaystyle\inf_{\mathcal{U}}I_{\alpha}(\theta,i,{\textbf{u}})\end{eqnarray}
where $I_{\alpha}(\theta,i,{\textbf{u}})$ is as in
\eqref{discounted}. The function $V_{\alpha}$ is called the
$\alpha-$discounted value function. Our aim is to characterise the
value function and to obtain an optimal control.

\noindent Instead of working with $V_{\alpha}$ we first start with
\begin{eqnarray}\label{exponential}W_{\alpha}(\theta ,i)=
\displaystyle \inf_{\mathcal{U}}\exp\bigl[\theta I_{\alpha}(\theta, i, {\textbf{u}})\bigr]\,.\end{eqnarray}
Formally, for any $T > 0$
{\small\begin{align*}W_{\alpha}(\theta, i)&= \inf_{\mathcal{U}}\mathbb{E}_i^{\textbf{u}}\biggl\{\exp\biggl[ \theta \int_0^Te^{-\alpha t}c(X_t,U_t)dt+\theta \int_T^\infty e^{-\alpha t}c(X_t,U_t)dt\biggr]\biggr\}\\&= \inf_{\mathcal{U}}\mathbb{E}_i^{\textbf{u}}\biggl\{\exp\biggl[ \theta \int_0^Te^{-\alpha t}c(X_t,U_t)dt\biggr]
\mathbb{E}_{X_T}^{\textbf{u}}\biggl(\exp\biggl[\theta e^{-\alpha T}\int_0^{\infty}e^{-\alpha t}c(X_t,U_t)dt\biggr]\biggl)\biggr\}\\
&=\inf_{\mathcal{U}}\mathbb{E}_i^{\textbf{u}}\biggl\{\exp\biggl[ \theta \int_0^Te^{-\alpha t}c(X_t,U_t)dt\biggr]W_{\alpha}(\theta e^{-\alpha T},X_T)\biggr\}\,.
\end{align*}}
If $W_{\alpha}(.,i)$ is smooth then, using Dynkin's formula and some heuristic arguments we obtain that $W_{\alpha}$ should satisfy
\begin{eqnarray}\label{dpeih}\begin{cases}\alpha \theta \frac{dW_{\alpha}}{d\theta}(\theta, i)= \displaystyle\inf_{u\in U_i}\biggl[\theta c(i,u) W_{\alpha}(\theta,i)+ \sum_{j}\lambda_{ij}(u)W_{\alpha}(\theta , j)\biggr]\\
\mbox{with}\,\,\displaystyle\lim_{\theta \rightarrow 0}W_{\alpha}(\theta,i)=1\,.
\end{cases}
\end{eqnarray}Equation \eqref{dpeih} is known as the HJB equation for the cost criterion \eqref{exponential}. Now starting with $\eqref{dpeih}$ the following verification theorem can be obtained.
\begin{thm}\label{verification}Assume that there exists a bounded, smooth (continuously differentiable in the first variable) function $w(\theta,i)$ such that
\begin{eqnarray}\label{ihdpe}\alpha \theta \frac{dw}{d\theta}(\theta, i)= \inf_{u\in U_i}
\biggl[\theta c(i,u) w(\theta,i)+ \sum_{j}\lambda_{ij}(u)w(\theta , j)\biggr]\,\,\mbox{on}\,\,(0,1)\times S
\end{eqnarray}and $w(\theta,i)\rightarrow 1$ as $\theta \rightarrow 0$ uniformly in $i$. Then $w(\theta,i)=W_{\alpha}(\theta, i)$.
Furthermore an optimal control for the cost criterion is given by \eqref{exponential}is given by
\begin{eqnarray}\label{control1}U^*_t=\textbf{u}^*(\theta e^{-\alpha t},X_{t-})\end{eqnarray}
where $\textbf{u}^*$ is given by
\begin{eqnarray}\label{func}\nonumber\displaystyle\inf_{u\in U_i}\bigl[\theta c(i,u)w(\theta,i) + \sum_{j}\lambda_{ij}(u)w(\theta,j)\bigr]=\\
 \bigl[\theta c(i,\textbf{u}^*(\theta,i)) w(\theta,i)+ \sum_{j}\lambda_{ij}(\textbf{u}^*(\theta,i))w(\theta,j)\bigr]\,.\end{eqnarray}
\end{thm}
\begin{proof}Define $\theta_t=\theta e^{-\alpha t}$ and
\begin{eqnarray*}\Psi _t=\exp\biggl\{\int_0^t\theta _s c(X_s,U_s)ds\biggr\}\end{eqnarray*}for any arbitrary Markov control $U_t=\textbf{u}(t,X_{t-})$. Then by Feynman - Kac formula we get
{\small\begin{align*}&\mathbb{E}_i^{\textbf{u}}\bigl\{\Psi _T w(\theta _T, X_T)\bigr\}-w(\theta, i)\\
&=\mathbb{E}_i^{\textbf{u}}\biggl\{\int_0^T \Psi_t \biggl[-\alpha \theta _t \frac{dw}{d \theta}(\theta _t, X_t)+ \theta _t c(X_t,U_t)w((\theta _t, X_t)) +\sum_{j}\lambda_{X_t j}(U_t)w(\theta_t, X_t)\biggr]dt\biggr\}\,.
\end{align*}}
Since $w$ satisfies \eqref{ihdpe}, the term on the righthand side above is non-negative. Therefore we get
$$w(\theta, i) \leq \mathbb{E}_i^{\textbf{u}}\bigl\{\Psi _T w(\theta _T, X_T)\bigr\}\,.$$ Now $\theta_T \rightarrow 0$ as $T \rightarrow \infty$ and hence $w(\theta_T, X_T )\rightarrow 1$. Thus we get
\begin{eqnarray}\label{inequality}w(\theta, i)\leq \mathbb{E}_i^{\textbf{u}}\biggl\{\exp\biggl[\theta \int_0^{\infty}e^{-\alpha t}c(X_t,U_t)dt\biggr]\biggr\}\,.\end{eqnarray} Similarly if we take the Markov control $U^*$ given by \eqref{control1} and \eqref{func} then we get equality in \eqref{inequality} in place of inequality. Hence the theorem follows.
\end{proof}
The following result is now evident.
\begin{cor}For $w$ as in Theorem \ref{verification}, define $v(\theta,i)=\theta^{-1}\log w(\theta,i)$. Then $v= V_{\alpha}$, where $V_{\alpha}$ is as in \eqref{infinitehorizon} .
\end{cor}
Now we prove that the HJB equation $\eqref{ihdpe}$ indeed has a smooth solution. To this end we first prove the following result.
\begin{prop}Let $\epsilon >0$ be arbitrary but fixed. There exists a function $W_{\epsilon}$ in  $C_b([\epsilon, 1)\times S)\bigcap C^1((\epsilon,1)\times S)$ such that $W_{\epsilon}$ satisfies
{\small\begin{align}\label{dpe1}\begin{cases}&\alpha \theta \frac{dW_{\epsilon}}{d\theta}(\theta, i)=\displaystyle\inf_{u\in U_i}\biggl[\theta c(i,u) W_{\epsilon}(\theta,i)+ \sum_{j}\lambda_{ij}(u)W_{\epsilon}(\theta , j)\biggr]\,\mbox{on}\,(\epsilon,1)\times S\\
&W_{\epsilon}(\epsilon,i)=e^{\frac{\epsilon}{\alpha}||c||}:=h_{\epsilon}(i)\,.
\end{cases}
\end{align}}
\end{prop}
\begin{proof}Let $\delta > 0$ which will be specified soon. Define $T:C_b([\epsilon, \epsilon+\delta]\times S)\rightarrow C_b([\epsilon, \epsilon+\delta]\times S)$ by
\begin{eqnarray*}T f(\eta,i)=e^{\frac{\epsilon}{\alpha}||c||}+ \frac{1}{\alpha}\int_{\epsilon}^{\eta}\inf_{u\in U_i}\bigl[c(i,u)f(\theta,i)+ \frac{1}{\theta}\sum_{j}\lambda_{ij}(u)f(\theta,j)\bigr]d\theta\,.
\end{eqnarray*}
Then
\begin{eqnarray*}|T f_1(\eta,i)- T f_2(\eta,i)| \leq \frac{1}{\alpha}\biggl[||c||\delta ||f_1-f_2||+\frac{2M}{\epsilon}\delta ||f_1-f_2||\biggr]\,,
\end{eqnarray*}where $M$ is as in \textbf{(A2)}. Choose $\delta$ such that $$\beta := \frac{1}{\alpha}\biggl[||c||\delta +\frac{2M}{\epsilon}\delta \biggr]$$ is strictly less than $1$. Then $T$ is a contraction. Hence by Banach's fixed point theorem there exists a $W$ in $C_b([\epsilon, \epsilon+\delta]\times S)$ which is the unique fixed point of $T$. Now assumptions \textbf{(A1)-(A3)} and the continuity of $c$ imply that $W$ is in $C_b([\epsilon, \epsilon+\delta]\times S)\bigcap C^1((\epsilon, \epsilon+ \delta]\times S)$. Thus it follows that $W$ satisfies $\eqref{dpe1}$ on $[\epsilon, \epsilon + \delta]\times S$. Proceeding in this way we will get a function $W_{\epsilon}\in C_b([\epsilon, 1)\times S)\bigcap C^1((\epsilon,1)\times S)$ which satisfies $\eqref{dpe1}$.
\end{proof}
Next we take limit $\epsilon \rightarrow 0$ of $W_{\epsilon}$ and show that the limit satisfies \eqref{ihdpe}. In particular we prove the following theorem:
\begin{thm}\label{maintheorem}Assume \textbf{(A1)-(A3)}and further assume that $S$ is finite.
Then there exists a unique solution $W$ in the class $C_b((0, 1)\times S)\bigcap C^1((0,1)\times S)$ to the equation
\begin{eqnarray*}\begin{cases}\alpha \theta \frac{dW}{d\theta}(\theta, i)=
\displaystyle\inf_{u\in U_i}\biggl[\theta c(i,u) W(\theta,i)+ \sum_{j}\lambda_{ij}(u)W(\theta , j)\biggr]\,\,\mbox{on} \,\, (0,1)\times S\\
\mbox{with}\,\,\displaystyle\lim_{\theta \rightarrow 0}W(\theta,i)=1\,.\end{cases}\end{eqnarray*}\end{thm}
\begin{proof} Using Dynkin's formula it can be shown that $W_{\epsilon}$ has the following stochastic representation:
\begin{eqnarray*}W_{\epsilon}(\theta,i)=\displaystyle\inf_{\mathcal{U}}
\mathbb{E}^{\textbf{u}}_i\biggl[h_{\epsilon}\exp\biggl(\theta \int_0^{T_{\epsilon}}e^{-\alpha t}c(X_t,U_t)dt\biggr)\biggr]
\end{eqnarray*}where $h_{\epsilon}$ is as in \eqref{dpe} and $T_{\epsilon}=\inf \{t\geq 0:\theta_t=\epsilon\}$,
i.e., $T_{\epsilon}=\frac{\log(\frac{\theta}{\epsilon})}{\alpha}$.
From this representation of $W_{\epsilon}$ we can deduce that for every $\epsilon > 0$
\begin{eqnarray*}0\leq W_{\epsilon}(\theta,i)\leq e^{\frac{\theta}{\alpha}||c||}\leq e^{\frac{||c||}{\alpha}}\,.
\end{eqnarray*}
Now we show that $\frac{dW_{\epsilon}}{d \theta}$ is also uniformly (in $\epsilon >0$) bounded.
For any arbitrary Markov control $\textbf{u}$,
\begin{align*}\bigg{|}&\mathbb{E}_i^{\textbf{u}}\bigl[h_{\epsilon}\exp\bigl((\theta + \delta)
\int_{0}^{T_{\epsilon}^{\delta}}e^{-\alpha t}c(X_t,U_t)dt\bigr)\bigr]-\mathbb{E}_i^{\textbf{u}}
\bigl[h_{\epsilon}\exp\bigl(\theta \int_{0}^{T_{\epsilon}}e^{-\alpha t}c(X_t,U_t)dt\bigr)\bigr]\bigg{|}\\
&\leq I_1+I_2
\end{align*}where $(\theta + \delta)e^{-\alpha T_{\epsilon}^{\delta}}=\epsilon$ and
{\small\begin{eqnarray*}I_1=\bigg{|}\mathbb{E}_i^{\textbf{u}}\bigl[h_{\epsilon}
\exp\bigl((\theta + \delta)\int_{0}^{T_{\epsilon}^{\delta}}e^{-\alpha t}c(X_t,U_t)dt\bigr)\bigr]-
\mathbb{E}_i^{\textbf{u}}\bigl[h_{\epsilon}\exp\bigl(\theta \int_{0}^{T_{\epsilon}^{\delta}}e^{-\alpha t}c(X_t,U_t)dt\bigr)\bigr]\bigg{|}\,,
\end{eqnarray*}}
{\small \begin{eqnarray*}I_2=\bigg{|}\mathbb{E}_i^{\textbf{u}}
\bigl[h_{\epsilon}\exp\bigl(\theta \int_{0}^{T_{\epsilon}^{\delta}}e^{-\alpha t}c(X_t,U_t)dt\bigr)\bigr]
-\mathbb{E}_i^{\textbf{u}}\bigl[h_{\epsilon}\exp\bigl(\theta \int_{0}^{T_{\epsilon}}e^{-\alpha t}c(X_t,U_t)dt\bigr)\bigr]\bigg{|}\,.
\end{eqnarray*}}
Now \begin{align*}I_1&\leq e^{\frac{||c||}{\alpha}}\mathbb{E}_i^{\textbf{u}} \biggl[\exp\bigl(\theta \int_{0}^{T_{\epsilon}^{\delta}}e^{-\alpha t}c(X_t,U_t)dt\bigr)\bigg{|}\exp\bigl(\delta \int_{0}^{T_{\epsilon}^{\delta}}e^{-\alpha t}c(X_t,U_t)dt\bigl)-1\bigg{|}\biggr]\\&\leq C_1e^{\frac{2||c||}{\alpha}}\delta \frac{||c||}{\alpha}\end{align*}for some constant $C_1>0$ and for $\delta >0$ small enough.
Similarly for $I_2$ we have
\begin{align*}I_2 &\leq e^{\frac{||c||}{\alpha}}\mathbb{E}_i^{\textbf{u}} \biggl[\exp\bigl(\theta \int_{0}^{T_{\epsilon}}e^{-\alpha t}c(X_t,U_t)dt\bigr)
\bigg{|}\exp\bigl(\theta \int_{T_{\epsilon}}^{T_{\epsilon}^{\delta}}e^{-\alpha t}c(X_t,U_t)dt\bigr)-1\bigg{|}\biggr]\\
&\leq e^{\frac{2||c||}{\alpha}}\biggl[e^{\frac{\theta ||c||}{\alpha}\bigl(e^{-\alpha T_{\epsilon}}-e^{-\alpha T_{\epsilon}^{\delta}}\bigr)}-1\biggr]\,.
\end{align*}
But $\theta e^{-\alpha T_{\epsilon}}-\theta e^{-\alpha T_{\epsilon}^{\delta}}=\delta e^{-\alpha T_{\epsilon}^{\delta}}=\frac{\epsilon \delta}{\theta + \delta }$ .
Hence we have
\begin{align*}I_2 &\leq e^{\frac{2||c||}{\alpha}}\biggl[e^{\frac{\theta ||c||}{\alpha}\frac{\epsilon \delta}{\theta + \delta }}-1\biggr]\\
&\leq C_2e^{\frac{2||c||}{\alpha}}\delta \frac{||c||}{\alpha}\end{align*} for some constant $C_2>0$ and for $\delta>0$ small enough.
Hence we can conclude that for $\delta > 0$, small enough
$$\big{|}W_{\epsilon}(\theta + \delta,i)- W_{\epsilon}(\theta,i)\big{|}\leq C_{3}e^{\frac{2||c||}{\alpha}}\delta \frac{||c||}{\alpha}$$for some constant $C_3>0$.

\noindent Similarly for $\delta < 0$, small enough, we can get an estimate of the type
$$\big{|}W_{\epsilon}(\theta + \delta,i)- W_{\epsilon}(\theta,i)\big{|}\leq C_{3}e^{\frac{2||c||}{\alpha}}|\delta| \frac{||c||}{\alpha}\,.$$
Thus we get that $\frac{d W_{\epsilon}}{d \theta}$ is uniformly bounded in $\epsilon >0$.

\noindent Now define
\begin{eqnarray*}\widetilde{W_{\epsilon}}(\theta,i)=\begin{cases} W_{\epsilon}(\theta ,i) \,\,\mbox{for}\,\,\theta > \epsilon\\
h_{\epsilon}(i) \, \, \mbox{for}\,\,\theta \leq \epsilon\,.
\end{cases}
\end{eqnarray*}
Then $\widetilde{W_{\epsilon}}$ satisfies the same bounds. Now since $\widetilde{W_{\epsilon}}$ is
uniformly bounded and $\frac{d\widetilde{W_{\epsilon}}}{d \theta}$ is also uniformly bounded,
by Ascoli - Arzela theorem there exists a function $W$ in $C_b((0,1)\times S)$ and a sequence
$\epsilon _n \rightarrow 0$ such that $\widetilde{W_{\epsilon_n}}\rightarrow W$ uniformly over
compact subsets of $(0,1)\times S$. Also by the definition of $\widetilde{W_{\epsilon}}$, $W(\theta,i)\rightarrow 1$ as $\theta \rightarrow 0$.
Now taking $\varphi \in C_c^{\infty}(0,1)$ we get
\begin{align*}&-\int_{0}^{1}\alpha \widetilde{W_{\epsilon_n}}\frac{d(\theta \varphi)}
{d \theta}d \theta =\int_0^{1}\alpha \theta \frac{d\widetilde{W_{\epsilon_n}}}
{d \theta}\varphi(\theta)d\theta \\&=\int_{0}^1\displaystyle\inf_{u\in U_i}\biggl[\theta c(i,u)
\widetilde{W_{\epsilon_n}}(\theta,i)+ \sum_{j}\lambda_{ij}(u)\widetilde{W_{\epsilon_n}}(\theta , j)\biggr]
\varphi(\theta)d \theta\\&-\int_{0}^{\epsilon_n}\inf_{u\in U_i}\biggl[\theta c(i,u) \widetilde{W_{\epsilon_n}}
(\theta,i)+ \sum_{j}\lambda_{ij}(u)\widetilde{W_{\epsilon_n}}(\theta , j)\biggr]\varphi(\theta)d \theta \,.
\end{align*}
Now taking limit $n \rightarrow \infty$ we get
\begin{align*}-\int_{0}^{1}\alpha W \frac{d (\theta \varphi)}
{d \theta}d \theta= \int_{0}^1\displaystyle\inf_{u\in U_i}\biggl[\theta c(i,u) W(\theta,i)+
\sum_{j}\lambda_{ij}(u)W(\theta , j)\biggr]\varphi(\theta)d \theta\,.
\end{align*}
Thus
\begin{eqnarray*}\alpha \theta \frac{d W}{d \theta}=\inf_{u\in U_i}\biggl[\theta c(i,u) W(\theta,i)+
\sum_{j}\lambda_{ij}(u)W(\theta , j)\biggr]\end{eqnarray*} in the sense of distribution.
But by our assumptions the righthand side is a continuous function. Therefore $\frac{d W}{d \theta}$ is in $C((0,1)\times S)$.
Thus $W$ is a smooth solution to the HJB equation {\eqref{dpe}}. The uniqueness follows from Theorem \ref{verification}.
\end{proof}
This immediately yields the following result:
\begin{thm}Assume \textbf{(A1)-(A3)} and that $S$ is finite.Then the value function $V_{\alpha}$  as in \eqref{infinitehorizon}
is the unique solution in $C_b((0, 1)\times S)\bigcap C^1((0,1)\times S)$ to
\begin{align*}&\alpha \theta \bigl[v + \theta \frac{d v}{d \theta}\bigr]e^{\theta v}
=\displaystyle \inf_{u\in U_i}\bigl[\theta c e^{\theta v}+ \sum_{j}\lambda_{ij}(u)
e^{\theta v(\theta,j)}\bigr] \, \mbox{on} \, (0,1)\times S\,\\ &\mbox{with} \,\,
\lim_{\theta \rightarrow 0}v(\theta,i)= \inf_{\mathcal{U}} \mathbb{E}_i^{\textbf{u}}
\int_0^{\infty}e^{-\alpha t}c(X_t,U_t)dt \,.\end{align*} An optimal control is given by the Markov control
$U_t=\textbf{u}^*(\theta e^{-\alpha t},X_{t-})$ where $\textbf{u}^*$ is given by
\begin{eqnarray*}  \displaystyle\inf_{u\in U_i}\bigl[\theta c(i,u)
e^{\theta V_{\alpha}}+ \sum_{j}\lambda_{ij}(u)e^{\theta V_{\alpha}(\theta,j)}\bigr]=\\
 \bigl[\theta c(i,\textbf{u}^*(\theta,i)) e^{\theta V_{\alpha}}+ \sum_{j}\lambda_{ij}
 (\textbf{u}^*(\theta,i))e^{\theta V_{\alpha}(\theta,j)}\bigr]\,.\end{eqnarray*}
\end{thm}
\begin{rem} The finiteness of the state space in Theorem
\ref{maintheorem} is forced upon by the uniformity in the boundary
condition in \ref{verification}. Note that the limiting procedure
that we have employed only yields that $\lim_{\theta \rightarrow
0}W(\theta,i)=1$ for each $i$. Hence the finiteness assumption on
$S$. Note that a similar situation arises in the risk-sensitive
control of diffusion processes \cite{Menaldi}. In \cite{Menaldi}
the authors treat periodic diffusions for which the state space is
a torus which is compact.
\end{rem}
\section{\textbf{Infinite Horizon Average Cost}}
In this section we study the infinite horizon average cost case. In order to study the average cost case we make some further assumptions on our model.

\noindent \textbf{(A4)} For every stationary control $\textbf{u}$, the corresponding Markov chain is irreducible.

\noindent \textbf{(A5)} There exists a Lyapunov function $V : S
\rightarrow \mathbb{R}^+$, an unbounded function \\ $W : S
\rightarrow [1, \infty)$ and constants $\delta > 0$ and $b <
\infty$ such that
\begin{eqnarray}\label{lyapunov}e^{-V(i)}\sum _j \lambda_{ij}(u)e^{V(j)} \leq -\delta W(i) + b 1_{\{0\}}(i) \,\, \mbox{for~all} \,\, i,u.\end{eqnarray}
An important consequence of \textbf{(A5)} is the following lemma:
\begin{lem}\label{drift} Let $\eta < \delta$ and \begin{eqnarray}\label{hitting}\tau_0 = \inf\{t>0 : X_t = 0\}\,.\end{eqnarray} Then
$$\sup _{\textbf{u}}\mathbb{E}_i^{\textbf{u}}e^{\eta \tau_0} \leq e^{V(i)} .$$
\end{lem}
\begin{proof}Let $$\tau_n = \sup\{t\geq 0 : X_t \leq n\}\,.$$If $X_0 \geq n$, then $\tau_n$ is $0$.
Assumption \textbf{(A5)} implies that there exists a $\tilde{b}$ and a $\widetilde{V}= e^{V} $ such that
\begin{eqnarray*}\sum _j \lambda_{ij}(u)\widetilde{V}(j) \leq -\delta \widetilde{V}(i) + \tilde{b} 1_{\{0\}}(i)\,.
\end{eqnarray*}
By Dynkin's formula we get
\begin{align*}\mathbb{E}_i^{\textbf{u}} e^{\eta(\tau_0 \wedge \tau_n \wedge n)}
\widetilde{V}(X_{\tau_0 \wedge \tau_n \wedge n}) &= \widetilde{V}(i) + \mathbb{E}_i^{\textbf{u}}
\int _0 ^{\tau_0 \wedge \tau_n \wedge n} e^{\eta s}(\Lambda^{\textbf{u}}+\eta)\widetilde{V}(X_s)ds\\
&\leq \widetilde{V}(i) + \mathbb{E}_i^{\textbf{u}}\int _0 ^{\tau_0 \wedge \tau_n \wedge n} e^{\eta s}(\eta - \delta)\widetilde{V}(X_s)ds\,.
\end{align*}
Thus we have
\begin{align*}\widetilde{V}(i) &\geq \mathbb{E}_i^{\textbf{u}} e^{\eta(\tau_0 \wedge \tau_n \wedge n)}
\widetilde{V}(X_{\tau_0 \wedge \tau_n \wedge n}) + \mathbb{E}_i^{\textbf{u}}\int _0 ^{\tau_0 \wedge \tau_n \wedge n}
e^{\eta s}(-\eta + \delta)\widetilde{V}(X_s)ds\\
&\geq \mathbb{E}_i^{\textbf{u}} e^{\eta(\tau_0 \wedge \tau_n \wedge n)}\,.
\end{align*}
Now letting $n\rightarrow \infty$ we get the desired result.
\end{proof}
Finally we make the following assumption

\noindent \textbf{(A6)} For $\tau_0$ as defined in \eqref{hitting}, $\displaystyle\sup_{i, \textbf{u}}\mathbb{E}_i^{\textbf{u}}\tau_0 < \infty$ .

\begin{rem} If the state space is finite then it can be easily seen that \textbf{(A5)} implies \textbf{(A6)}.
\end{rem}
Now we state and prove the main theorem of this section:
\begin{thm}Under \textbf{(A1)}-\textbf{(A6)},
an optimal control for the risk-sensitive average cost criterion exists for
$\theta$ and $c$ satisfying $\theta ||c|| < \delta$ where $\delta$ is as in $\eqref{lyapunov}$.
\end{thm}
\begin{proof}Let \begin{eqnarray}\label{eigenvalue}\theta \rho^{\textbf{u}}=
\lim _{T \rightarrow \infty}\frac{1}{T}\log \mathbb{E}_{i}^{\textbf{u}}
\biggl[\exp\biggl(\theta \int_0^T c(X_t,\textbf{u}(X_t))dt\biggr)\biggr]\,. \end{eqnarray}
The existence of the above limit follows from the multiplicative ergodic theorems proved in \cite{Meyn1} and \cite{Meyn2} .
Moreover it also follows from the results in \cite{Meyn1} and \cite{Meyn2} that the limit in \eqref{eigenvalue} is the
principal eigenvalue for the operator $\Lambda^{\textbf{u}}+ \theta c$ and has a positive eigenfunction which
belongs to the class $L^{\infty}_{\widetilde{V}}$, i.e., if we denote an eigenfunction by $h^{\textbf{u}}$
then $\sup\frac{|h^{\textbf{u}}(i)|}{\widetilde{V}(i)}< \infty$. Thus the following equation holds
\begin{eqnarray}\label{Poisson eqn}\sum_{j}\lambda_{ij}(\textbf{u}(i))h^{\textbf{u}}(j) + \theta c(i,\textbf{u}(i))h^{\textbf{u}}(i)
 = \rho^{\textbf {u}} \theta h^{\textbf{u}}(i)\,.
\end{eqnarray}Equation \eqref{Poisson eqn} is referred to as the Poisson equation.
Now it is clear that if $h^{\textbf{u}}$ satisfies \eqref{Poisson eqn} then so does any scalar multiple
of $h^{\textbf{u}}$. Therefore without any loss of generality we may assume that $h^{\textbf{u}}(0)=1$.
With this restriction, using Dynkin's formula and the fact that $h^{\textbf{u}}$ satisfies \eqref{Poisson eqn}
we get the following stochastic representation for $h^{\textbf{u}}$:
\begin{align}h^{\textbf{u}}(i)=\mathbb{E}_i^{\textbf{u}}\biggl[\exp\biggl(\theta \int _0^{\tau_0}
(c(X_s,\textbf{u}(X_s))-\rho^{\textbf{u}})ds\biggr)\biggr]\,.
\end{align}
Now using the stochastic representation of $h^{\textbf{u}}$ we
derive bounds on $h^{\textbf{u}}$. First we derive an upper bound. We have
\begin{align*}h^{\textbf{u}}(i)&=\mathbb{E}_i^{\textbf{u}}
\biggl[\exp\biggl(\theta \int _0^{\tau_0}(c(X_s,\textbf{u}(X_s))-\rho^{\textbf{u}})ds\biggr)\biggr]\\
&\leq \mathbb{E}_i^{\textbf{u}}e^{\theta ||c|| \tau_0}\\
&\leq e^{V(i)}
\end{align*}
by Lemma \ref{drift}, provided $\theta ||c|| < \delta$.\\ This upper bound shows that bound on $h^{\textbf{u}}$ is uniform in $\textbf{u}$.
Next we obtain a lower bound. We have
\begin{align*}h^{\textbf{u}}(i)&=\mathbb{E}_i^{\textbf{u}}
\biggl[\exp\biggl(\theta \int _0^{\tau_0}(c(X_s,\textbf{u}(X_s))-\rho^{\textbf{u}})ds\biggr)\biggr]\\
& \geq \exp \biggl\{\mathbb{E}_i^{\textbf{u}}\theta \int _0^{\tau_0}(c(X_s,\textbf{u}(X_s))-\rho^{\textbf{u}})ds\biggr\}\\
&\geq \exp(-\theta \rho^{\textbf{u}}\mathbb{E}_i^{\textbf{u}}\tau_0)\\
&\geq \exp(-\theta ||c||\mathbb{E}_i^{\textbf{u}}\tau_0)\\
&> \epsilon
\end{align*}for some $\epsilon > 0$. In the above sequence of
inequalities the second one follows from Jensen's inequality and the last one follows from \textbf{(A6)}.\\
 Let \begin{eqnarray}\label{optimal}\rho^*=\displaystyle\inf_{\textbf{u}}\rho^{\textbf{u}}\,.\end{eqnarray}
 Next we show that there exists a control $\textbf{u}^*$ which attains the infimum in \eqref{optimal}.
 From \eqref{optimal} it follows that there exists a sequence $\textbf{u}_n$ such that $\rho^{\textbf{u}_n}\rightarrow \rho^*$.
 Since each $U_i$ is compact there exists a subsequence which is also denoted by $\textbf{u}_n$ again and a $\textbf{u}^*$ such that
$$\textbf{u}_n\rightarrow \textbf{u}^*\,\,\mbox{pointwise}\,.$$
Again since $h^{\textbf{u}_n}$ is pointwise bounded, there exists a subsequence which we call $h^{\textbf{u}_n}$ again such that
$$h^{\textbf{u}_n}(i)\rightarrow h^*(i) \,\, \mbox{for~each}\,\,i\,,$$ for some $h^*$ and $\displaystyle\inf_i h^*(i)\geq \epsilon$.
Therefore by using Fatou's lemma we have
\begin{align*}\sum_{j\neq i}\lambda_{ij}(\textbf{u}^*(i))h^*(j)&\leq \liminf_{n\rightarrow \infty}
\sum_{j\neq i}\lambda_{ij}(\textbf{u}_n(i))h^{\textbf{u}_n}(j)\\
&=\liminf_{n \rightarrow \infty}[-\lambda_{ii}(\textbf{u}_n(i))
h^{\textbf{u}_n}(i)-\theta c(i,\textbf{u}_n(i))h^{\textbf{u}_n}(i)+ \theta \rho^{\textbf{u}_n}h^{\textbf{u}_n}(i)]\\
&=-\lambda_{ii}(\textbf{u}^*(i))h^*(i)-\theta c(i,\textbf{u}^*(i))h^*(i)+ \theta \rho^*h^*(i)\,.
\end{align*}
Thus we get
\begin{align*}\sum_j\lambda_{ij}(\textbf{u}^*(i))h^*(j)+\theta c(i,\textbf{u}^*(i))h^*(i)\leq \theta \rho^*h^*(i)\,.
\end{align*}
Now we claim that $$\rho^* = \rho^{\textbf{u}^*}\,.$$
Indeed, with $\tau_n$ as in the proof of Lemma \ref{drift} we get from Dynkin's formula
\begin{align*}\mathbb{E}_i^{\textbf{u}^*}&\biggl[\exp\biggl(\theta \int_0^{T\wedge \tau_n}c(X_s,\textbf{u}^*(X_s))ds\biggr)
h^*(X_{T\wedge\tau_n})\biggr]\\&=h^*(i)+\mathbb{E}_i^{\textbf{u}^*}\biggl[\int_0^{T\wedge\tau_n}
\exp\bigl(\theta \int_0^tc(X_s, \textbf{u}^*(X_s))ds\bigr)(\Lambda^{\textbf{u}^*}+\theta c)h^*(X_t)dt\biggr]\\
&\leq h^*(i) + \theta \rho^* \mathbb{E}_i^{\textbf{u}^*}\biggl[\int_0^{T\wedge\tau_n}
\exp\bigl(\theta \int_0^tc(X_s, \textbf{u}^*(X_s))ds\bigr)h^*(X_t)dt\biggr]\\
&\leq h^*(i) +  \theta \rho^* \int_0^T \mathbb{E}_i^{\textbf{u}^*}
\biggl[\exp\bigl(\theta \int_0^{t\wedge \tau_n}c(X_s, \textbf{u}^*(X_s))ds\bigr)h^*(X_{t\wedge\tau_n})\biggr]dt\,.
\end{align*}
Hence by Gronwall's inequality we have
\begin{align*}\mathbb{E}_i^{\textbf{u}^*}&\biggl[\exp\biggl(\theta \int_0^{T\wedge \tau_n}
c(X_s,\textbf{u}^*(X_s))ds\biggr) h^*(X_{T\wedge\tau_n})\biggr] \leq h^*(i)e^{\theta \rho^* T}\,.
\end{align*}
Therefore letting $n \rightarrow \infty$ we have
\begin{align*}h^*(i)e^{\theta \rho^* T}\geq \epsilon \mathbb{E}_i^{\textbf{u}^*}
&\biggl[\exp\biggl(\theta \int_0^T c(X_s,\textbf{u}^*(X_s))ds\biggr)\biggr]\,.
\end{align*}
which implies
\begin{align*}\theta \rho^* \geq \lim_{T \rightarrow \infty}\frac{1}{T}\log \mathbb{E}_i^{\textbf{u}^*}
&\biggl[\exp\biggl(\theta \int_0^T c(X_s,\textbf{u}^*(X_s))ds\biggr)\biggr]= \theta \rho^{\textbf{u}^*}.
\end{align*}
Hence $\rho^*=\rho^{\textbf{u}^*}$.
\end{proof}
\section{\textbf{Policy Improvement Algorithm}}
In the previous section we have proved the existence of an optimal control.
But our theorem is purely existential and does not give an algorithm to find an optimal control.
In this section we focus on the computational approach for finding an optimal stationary control.
Since we are concerned with algorithm in this section we assume that both the state and action spaces are finite.
Now we describe the policy improvement algorithm.\\
\textbf{Algorithm}\\
\textbf{Step 1:} Start with any initial stationary control $\textbf{u}_1$. For this  $\textbf{u}_1$
$$\rho^{\textbf{u}_1}= \lim _{T \rightarrow \infty}\frac{1}{\theta T}\log \mathbb{E}_{i}^{\textbf{u}_1}
\biggl[\exp\biggl(\theta \int_0^T c(X_t,\textbf{u}_1(X_t))dt\biggr)\biggr]$$and
$$h^{\textbf{u}_1}(i) = \mathbb{E}_{i}^{\textbf{u}_1}\biggl[\exp\biggl(\theta \int_0^{\tau_0} (c(X_t,\textbf{u}_1(X_t))-\rho^{\textbf{u}_1})dt\biggr)\biggr]\,.$$
We know from previous section that $h^{\textbf{u}_1}$ satisfies the Poisson equation
\begin{eqnarray*}\sum_{j}\lambda_{ij}(\textbf{u}_1(i))h^{\textbf{u}_1}(j)+ \theta c(i, \textbf{u}_1(i))
h^{\textbf{u}_1}(i)=\theta \rho^{\textbf{u}_1}h^{\textbf{u}_1}(i)
\end{eqnarray*} satisfying the constraint $h^{\textbf{u}_1}(0)=1$.\\
\textbf{Step 2:} Define $\textbf{u}_2$ to be the stationary control which minimizes
$$\min_{u\in U_i}\bigl[\theta c(i,u) h^{\textbf{u}_1}(i) + \sum_j \lambda_{ij}(u)h^{\textbf{u}_1}(j)\bigr]\,.$$
With $\rho^{\textbf{u}_2}$ and $h^{\textbf{u}_2}$ as above continue the procedure.
\begin{thm}\label{algo}The above algorithm leads to an optimal control in finite number of steps.
\end{thm}
\begin{proof}
In order to prove that this algorithm comes up with an optimal control in a finite number of steps we first claim that
\begin{eqnarray}\label{compare}\rho^{\textbf{u}_{n+1}}\leq \rho^{\textbf{u}_n}\,.\end{eqnarray}
Indeed, from the definition of $\textbf{u}_{n+1}$ we have
\begin{align*}\sum_{j}\lambda_{ij}(\textbf{u}_{n+1}(i))h^{\textbf{u}_{n}}(j)+ \theta c(i, \textbf{u}_{n+1}(i))h^{\textbf{u}_{n}}(i)&\leq
\sum_{j}\lambda_{ij}(\textbf{u}_n(i))h^{\textbf{u}_n}(j)+ \theta c(i, \textbf{u}_n(i))h^{\textbf{u}_n}(i)\\&=\theta \rho^{\textbf{u}_n}h^{\textbf{u}_{n}}(i)\,.
\end{align*}
Now using arguments involving Dynkin's formula as in the previous section it can be proved that $\rho^{\textbf{u}_{n+1}}\leq \rho^{\textbf{u}_n}$.\\
Our second claim is that, suppose for some $n$ and for all $i$
{\small\begin{eqnarray}\label{dpe} \sum_{j}\lambda_{ij}(\textbf{u}_n(i))h^{\textbf{u}_n}(j)+ \theta c(i, \textbf{u}_n(i))h^{\textbf{u}_n}(i)
= \sum_{j}\lambda_{ij}(\textbf{u}_{n+1}(i))h^{\textbf{u}_{n}}(j)+ \theta c(i, \textbf{u}_{n+1}(i))h^{\textbf{u}_{n}}(i)\,
\end{eqnarray}}then $\textbf{u}_n$ is optimal.\\
To prove this observe that if $\textbf{u}_n$ is as in \eqref{dpe} then
\begin{align}\label{best}\nonumber\theta \rho^{\textbf{u}_n}h^{\textbf{u}_{n}}(i)&=
\sum_{j}\lambda_{ij}(\textbf{u}_n(i))h^{\textbf{u}_n}(i)+ \theta c(i, \textbf{u}_n(i))h^{\textbf{u}_n}(i)\\
\nonumber&= \sum_{j}\lambda_{ij}(\textbf{u}_{n+1}(i))h^{\textbf{u}_{n}}(i)+ \theta c(i, \textbf{u}_{n+1}(i))h^{\textbf{u}_{n}}(i)\\
&= \min_{u\in U_i}\bigl[\theta c(i,u) h^{\textbf{u}_n}(i) + \sum_j \lambda_{ij}(u)h^{\textbf{u}_n}(i)\bigr]\,.
\end{align}
Now for any stationary control $\textbf{u}$ we have by Dynkin's formula
\begin{align*}&\mathbb{E}_i^{\textbf{u}}\biggl[\exp\biggl(\theta\int_0^T(c(X_t,\textbf{u}(X_t))
-\rho^{\textbf{u}_n})dt\biggr)h^{\textbf{u}_{n}}(X_T)\biggr]
\\&= h^{\textbf{u}_{n}}(i) + \mathbb{E}_i^{\textbf{u}}\int_0^T \exp\biggl(\theta\int_0^t(c(X_s,\textbf{u}(X_s))-\rho^{\textbf{u}_n})ds\biggr)[\Lambda^{\textbf{u}}+\theta c - \theta \rho^{\textbf{u}_n}]h^{\textbf{u}_{n}}(X_t)dt\\&\geq h^{\textbf{u}_{n}}(i)\,.
\end{align*}The last inequality follows from \eqref{best}. Thus we have
\begin{align*}h^{\textbf{u}_{n}}(i)\leq \max_ih^{\textbf{u}_{n}}(i)\mathbb{E}_i^{\textbf{u}}
\biggl[\exp\biggl(\theta\int_0^T(c(X_t,\textbf{u}(X_t))-\rho^{\textbf{u}_n})dt\biggr)\biggr]\,.
\end{align*}This implies that
\begin{align*}\rho^{\textbf{u}_n} \leq \lim _{T \rightarrow \infty}\frac{1}{\theta T}\log \mathbb{E}_{i}^{\textbf{u}}\biggl[\exp\biggl(\theta \int_0^T c(X_t,\textbf{u}(X_t))dt\biggr)\biggr]=\rho^{\textbf{u}}\,.
\end{align*}Hence the claim.\\
Our final claim is that if $\textbf{u}_n$ is not an optimal control then the inequality in \eqref{compare} is actually a strict inequality.\\
Since $\textbf{u}_n$ is not optimal, we have
\begin{align*}\sum_{j}\lambda_{ij}(\textbf{u}_{n+1}(i))h^{\textbf{u}_{n}}(i)+ \theta c(i, \textbf{u}_{n+1}(i))h^{\textbf{u}_{n}}(i)-\theta \rho^{\textbf{u}_n}h^{\textbf{u}_{n}}(i)= -g(i)\,,
\end{align*} where $g$ is a non-negative function and there exists at least one $i$ such that $g(i)>0$. Therefore for any $T>0$ it follows from Dynkin's formula that
\begin{align}\label{crucial}\nonumber&\mathbb{E}^{\textbf{u}_{n+1}}_i\biggl[\exp\biggl(\theta \int_0^T (c(X_s,\textbf{u}_{n+1}(X_s))-\rho^{\textbf{u}_n})ds\biggr)h^{\textbf{u}_n}(X_T)\biggr]\\\nonumber&= h^{\textbf{u}_n}(i)+ \mathbb{E}^{\textbf{u}_{n+1}}_i \biggl[\int_0^T \exp\biggl(\theta \int_0^t (c(X_s,\textbf{u}_{n+1}(X_s))-\rho^{\textbf{u}_n})ds\biggr)\biggl(\theta c_{n+1}+ \Lambda^{n+1}-\theta \rho^{\textbf{u}_n}\biggr)h^{\textbf{u}_n}(X_t)dt\biggr]\\\nonumber&=h^{\textbf{u}_n}(i)-\mathbb{E}^{\textbf{u}_{n+1}}_i \biggl[\int_0^T \exp\biggl(\theta \int_0^t (c(X_s,\textbf{u}_{n+1}(X_s))-\rho^{\textbf{u}_n})ds\biggr)g(X_t)dt\biggr]
\\\nonumber&=h^{\textbf{u}_n}(i)-\int_0^T\exp-(\rho^{\textbf{u}_n}-\rho^{\textbf{u}_{n+1}})t\,\exp\biggl(\theta \int_0^t (c(X_s,\textbf{u}_{n+1}(X_s))-\rho^{\textbf{u}_{n+1}})ds\biggr)g(X_t)dt
\\&=h^{\textbf{u}_n}(i)- h^{\textbf{u}_{n+1}}(i)T\frac{1}{T}\mathbb{\widetilde{E}}_{i}^{\textbf{u}_{n+1}}\int_0^T \frac{g(X_t)}{h^{\textbf{u}_{n+1}}(X_t)}dt
\end{align}if $\rho^{\textbf{u}_{n+1}}=\rho^{\textbf{u}_n}$. In \eqref{crucial} the expectation operator $\mathbb{\widetilde{E}}^{\textbf{u}_{n+1}}_i$ is given by
\begin{eqnarray}\label{twisted}\mathbb{\widetilde{E}}_{i}^{\textbf{u}_{n+1}}f(X_t)=\mathbb{E}^{\textbf{u}_{n+1}}_i\biggl[\exp\biggl(\theta \int_0^t (c(X_s,\textbf{u}_{n+1}(X_s))-\rho^{\textbf{u}_{n+1}})ds\biggr)\frac{h^{\textbf{u}_{n+1}}(X_t)}{h^{\textbf{u}_{n+1}}(i)}f(X_t)dt\biggr]\,\end{eqnarray} for any real valued bounded function $f$. It is easy to see that \eqref{twisted} uniquely determines a transition probability kernel $\mathbb{\widetilde{P}}^{\textbf{u}_{n+1}}_i$ and under $\mathbb{\widetilde{P}}^{\textbf{u}_{n+1}}_i$ the corresponding Markov chain is still irreducible. Since the state space is finite, the Markov chain under $\mathbb{\widetilde{P}}^{\textbf{u}_{n+1}}_i$ is positive recurrent. Thus it has a unique invariant measure, say, $\widetilde{\pi}$. Now observe that the righthand side of \eqref{crucial} is negative for $T$ sufficiently large because $$\frac{1}{T}\mathbb{\widetilde{E}}_{i}^{\textbf{u}_{n+1}}\int_0^T \frac{g(X_t)}{h^{\textbf{u}_{n+1}}(X_t)}dt \rightarrow \sum \widetilde{\pi}(i)\frac{g(i)}{h^{\textbf{u}_{n+1}}(i)}>0\,.$$  But the lefthand side is always non-negative. Thus we get a contradiction and hence $$\rho^{\textbf{u}_{n+1}}<\rho^{\textbf{u}_n}\,.$$

\noindent From the above claims it follows that the algorithm comes up with the optimal control within a finite number of steps because the number of controls is finite.
\end{proof}
Some comments are now in order.
\begin{rem}Suppose the state and action spaces are finite. Let $\textbf{u}^*$ be an optimal control. Let $\rho^*$ be the optimal average cost and $$h^{\textbf{u}^*}(i)=\mathbb{E}_{i}^{\textbf{u}^*}\biggl[\exp\biggl(\theta \int_0^{\tau_0} (c(X_t,\textbf{u}^*(X_t))-\rho^*)dt\biggr)\biggr]\,.$$Then it follows from the arguments used in the proof of Theorem \ref{algo} that $(\rho^*,h^{\textbf{u}^*})$ satisfies the equation
\begin{eqnarray}\label{acdpe}\theta \rho^*h^{\textbf{u}^*}(i)=\min_u\bigl[\theta c(i,u) h^{\textbf{u}^*}(i) + \sum_j \lambda_{ij}(u)h^{\textbf{u}^*}(i)\bigr]\,.
\end{eqnarray}Equation \eqref{acdpe} is the HJB equation for the average cost criterion. If $(\lambda, h)$ is a solution of \eqref{acdpe} where $h$ is a positive function then using Dynkin's formula it can be shown that $\lambda$ is the optimal cost and the minimiser in \eqref{acdpe} is an optimal control.
\end{rem}
\begin{rem}If the state space is countably infinite and equation \eqref{acdpe} has a solution $(\lambda, h)$ such that $h$ is a bounded, positive function which is uniformly bounded away from $0$, then again it can be shown that $\lambda$ is the optimal cost and the minimiser in \eqref{acdpe} is an optimal control. However, we have not been able to show that \eqref{acdpe} has such a solution. If one assumes that \eqref{acdpe} has such a solution then one can develop value and policy iteration algorithm along the lines of \cite{Borkar}. In \cite{Borkar} the authors deal with discrete time Markov chains. There they have developed value and policy iteration algorithm under the assumption that analogous dynamic programming equation has a solution.
\end{rem}
\section{\textbf{Conclusion}}In this paper we have studied risk-sensitive optimal control problem for continuous time Markov chains. We have analysed the finite horizon case under fairly general conditions. For the infinite horizon discounted cost case we have assumed that the state space is finite. So it will be interesting to investigate the problem for the case of countably infinite state space. The average cost case has been studied under an additional Lyapunov type stability condition. We have established the existence of an optimal control. We have also developed policy iteration algorithm for the case of finite state and action spaces. For countable state space an algorithmic approach to determine an optimal control needs further investigation. \\
\textbf{Acknowledgement :} The authors wish to thank V. S. Borkar for helpful discussions.


\begin{thebibliography}{99}

\bibitem{Ari} A. Arapostathis, V. S. Borkar and M. K. Ghosh, \textit{Ergodic Control of Diffusion Processes}, Encyclopedia of Mathematics and its Applications 143, Cambridge University Press, Cambridge, 2012.
\bibitem{benes} V. E. Benes, \textit{Existence of optimal strategies based on specified information for a class of stochastic decision problems}, SIAM J. Control 8 (1970), 179-188.
\bibitem{Biswas} A. Biswas, V. S. Borkar and K. S. Kumar, \textit{Risk-sensitive control with near monotone cost}, Appl. Math. Optim. 62 (2010), 145-163.
\bibitem{Pliska1} T. R. Bielecki and S. R. Pliska, \textit{Risk-sensitive dynamic asset management}, Appl. Math. Optim. 39 (1999), 337-360.
\bibitem{Pliska2} T. R. Bielecki, S. R. Pliska and S. J. Sheu, \textit{Risk sensitive portfolio management with Cox-Ingersoll-Ross interst rates: the HJB equation}, SIAM J. Control Optim 44 (2005), 1811-1843.
\bibitem{Borkar} V. S. Borkar and S. P. Meyn, \textit{Risk-sensitive optimal control for Markov decision processes with monotone cost}, Math. Oper. Res. 27 (2002), 192-209.
\bibitem{Chung} K. J. Chung and M. J. Sobel, \textit{Discounted MDPs: distribution functions and exponential utility maximization}, SIAM J. Control Optim. 25 (1987), 49-62.
\bibitem{Ethier} S. N. Ethier and T. G. Kurtz, \textit{Markov Processes : Characterization and
Convergence}, John Wiley and Sons, 1986.
\bibitem{Fleming} W. H. Fleming and W. M. McEneaney, \textit{Risk sensitive control on an infinite horizon}, SIAM J. Control Optim. 33 (1995), 1881-1915.
\bibitem{Goswami} M. K. Ghosh, A. Goswami and S. K. Kumar, \textit{Portfolio Optimization in a semi-Markov modulated market}, Appl. Math. Optim. 60 (2009), 275-296.
\bibitem{hgbook} X. Guo and O. Hern\'{a}ndez-Lerma, \textit{Continuous-Time Markov Decision Processes. Theory and Applications},
Springer-Verlag, 2009.
\bibitem{Prieto} X. Guo, O. Hern\'{a}ndez-Lerma and T. Prieto-Rumeau, \textit{A survey of recent
results on continuous-time Markov decision processes}, TOP 14 (2006), 177-261.
\bibitem{Lerma} O. Hern\'{a}ndez-Lerma, \textit{Adaptive Markov Control Processes}, Springer-Verlag, New York, 1989.
\bibitem{Hernandez} D. Hern$\acute{a}$ndez-Hern$\acute{a}$ndez and S. I. Marcus, \textit{Risk sensitive control of Markov processes in countable state space}, Systems Control Lett. 29 (1996), 147-155.
\bibitem{Howard} R. A. Howard and J. E. Matheson, \textit{Risk-sensitive Markov decision processes}, Mananagement Sci. 18 (1972), 356-369.
\bibitem{Jacobson} D. H. Jacobson \textit{Optimal stochastic linear systems with exponential performance criteria and their relation to deterministic differential games}, IEEE Trans. on Automat. Control 18 (1973), 124-131.
\bibitem{JasKiewicz} A. Ja$\acute{s}$kiewicz, \textit{Average optimality for risk-sensitive control with general state space}, Ann. Appl. Probab. 17 (2007), 654-675.
\bibitem{Meyn1} I. Kontoyiannis and S. P. Meyn, \textit{Spectral theory and limit theorems for geometrically ergodic Markov processes}, Ann. Appl. Probab. 13 (2003), 304-362.
\bibitem{Meyn2} I. Kontoyiannis and S. P. Meyn, \textit{Large deviations asymptotics and the spectral theory of multiplicatively regular Markov Processes}, Electron. J. Probab. 10 (2005), 61-123.
\bibitem{Markowitz} H. Markowitz, \textit{Portfolio Selection}. J. of Finance 7 (1952), 77-91.
\bibitem{dimasi} G. B. Di Masi and L. Stettner, \textit{Risk-senstive control of discrete time Markov processes with infinite horizon}, SIAM J. Control Optim. 38 (1999), 61-78.
\bibitem{Menaldi} J. L. Menaldi and M. Robin, \textit{Remarks on risk sensitive control problems}, Appl. Math. Optim. 52 (2005), 297-310.
\bibitem{Nagai} H. Nagai, \textit{Bellman equations of risk-sensitive control}, SIAM J. Control Optim. 34 (1996), 74-101.
\bibitem{Sobel} G. E. Monahan and M. J. Sobel, \textit{Risk-sensitive dynamic market share attraction games}, Games Econom.Behav. 20 (1997), 149-160.
\bibitem{PR} T. Prieto-Romeau and O. Hern\'{a}ndez-Lerma, \textit{Variance minimization and the overtaking optimality
approach to continuous-time controlled Markov chains}, Math. Meth.
Oper. Res. 70 (2009), 527-540.
\bibitem{Sharpe} W. F. Sharpe, \textit{Capital asset prices: A theory of market equilibrium under conditions of risk}, J. of Finance 19 (1964), 425-442.
\bibitem{Whittle} P. Whittle, \textit{Risk-Sensitive Optimal Control}, Wiley, New York, 1990.
\end{thebibliography}
\end{document}